\documentclass[11pt,a4paper]{article}
\usepackage{amsmath}
\usepackage{amssymb}
\usepackage{amsthm}
\usepackage{geometry}
\usepackage[colorlinks=true,citecolor=black,linkcolor=black,urlcolor=black]{hyperref}
\usepackage{amsrefs}
\usepackage[english]{babel}
\theoremstyle{plain}
\newtheorem{theorem}{Theorem}[section]
\newtheorem{proposition}[theorem]{Proposition}
\newtheorem{lemma}[theorem]{Lemma}
\newtheorem{corollary}[theorem]{Corollary}
\numberwithin{equation}{section}
\newcommand{\cO}{\mathcal O}
\DeclareMathOperator{\Nm}{N}
\DeclareMathOperator{\vol}{vol}
\DeclareMathOperator{\covol}{covol}
\title{Covering Projective Height Balls by Subspaces in Rigid Adelic Spaces}
\author{Ruida Di, Runjie Hu}
\date{}
\newcommand{\Addresses}{{
  \bigskip
  \footnotesize

   Ruida Di, \textsc{Morningside Center of Mathematics, Academy of Mathematics and Systems Science, Chinese Academy of Sciences, No.~55 Zhongguancun East Road, Beijing 100190, China.}\par\nopagebreak
  \textit{E-mail address:\ }\texttt{drdmath@amss.ac.cn}

  Runjie Hu, \textsc{Department of Mathematics, 3368 Texas A\&M University, College Station, TX 77843, USA.}\par\nopagebreak
  \textit{E-mail address:\ }\texttt{runjie.hu@tamu.edu}
}}
\begin{document}
\maketitle
\begin{abstract}
Let $E$ be an $n$-dimensional rigid adelic space over a number field $K$. We study the minimum number $g_E(R)$ of proper $K$-subspaces needed to cover the projective height ball of radius $R$, together with the maximum cardinality $h_E(R)$ of a subset in linear general position. We show that, once $R$ is sufficiently large compared with the last Roy--Thunder minimum of $E$, both quantities have order $\Psi_E(R)^{[K:\mathbb Q]}$, where $\Psi_E(R)$ is an explicit expression in the Roy--Thunder minima. The comparison constants are effectively computable and uniform in $E$. For the standard adelic space $K^n$, this gives order $R^{[K:\mathbb Q]n/(n-1)}$.
\end{abstract}



\section{Introduction}

The goal of this paper is to study coverings of projective points of bounded height by proper linear subspaces over number fields. We also consider the complementary extremal problem of finding large subsets in linear general position.

The corresponding lattice problem was studied by B{\'a}r{\'a}ny, Harcos, Pach, and Tardos \cite{BHPT}*{Theorems~1 and~2}. Let $\mathcal C\subset\mathbb R^n$ be a compact convex body symmetric about the origin. For $1\le i\le n$, define its $i$-th successive minimum with respect to $\mathbb Z^n$ by
\[
\lambda_i
=
\inf\left\{
\lambda>0:
\lambda\mathcal C\text{ contains $i$ linearly independent vectors of }\mathbb Z^n
\right\}.
\]
Thus $\lambda_1\le\cdots\le\lambda_n$. They compared the minimum number of proper linear subspaces needed to cover $\mathcal C\cap\mathbb Z^n$ with the largest cardinality of a subset in which every $n$ distinct points are linearly independent. When $\lambda_n$ is bounded away from $1$, both quantities are controlled, up to constant factors, by
\[
\min_{1\le m<n}(\lambda_m\cdots\lambda_n)^{-1/(n-m)}.
\]
The factor $1-\lambda_n$ in their lower bound shows that the boundary regime $\lambda_n\to1$ requires separate attention. Related covering problems for lattice points were studied by Bezdek--Hausel \cite{BezdekHausel}, Bezdek--Litvak \cite{BezdekLitvak}, Balko--Cibulka--Valtr \cite{BalkoCibulkaValtr}, and Sudakov--Tomon \cite{SudakovTomon}.

On the arithmetic side, the counting of points of bounded height in projective space goes back to Schanuel \cite{Schanuel}; see also Christensen--Gubler \cite{ChristensenGubler} and Ange \cite{Ange} for relative and adelic variants. Fukshansky \cite{FukshanskyAdditional} studied a complementary avoidance problem: given a subspace over a number field and finitely many prescribed proper subspaces, he obtained an explicit small-height point outside their union, using a counting estimate in adelic cubes. Gaudron developed adelic geometry-of-numbers estimates and generalized Siegel lemmas \cite{Gaudron2009}.

Our problem is different from these avoidance results. The proper subspaces are not prescribed in advance. Instead, we ask for the minimum number of proper $K$-subspaces whose union contains the entire projective height ball, and simultaneously for the largest subset of that ball in linear general position. Moreover, a projective height ball over a number field is not obtained by intersecting a single lattice with one fixed convex body. Finite-place integrality, ideal classes, and the unit group enter together. The Roy--Thunder minima provide the analogue of successive minima, but additional arithmetic input is needed to pass between height bounds and integral representatives.

Let $K$ be a number field of degree $d=[K:\mathbb Q]$, and let $E$ be an $n$-dimensional rigid adelic space over $K$. Our main result shows that, once the radius is sufficiently large compared with $\Lambda_n(E)$, the covering number and the general-position extremal quantity have the same order. The common scale is an explicit minimum of expressions involving all the Roy--Thunder minima, and the comparison constants are effective and uniform over $E$.

\subsection{Main results}

Let $K$ have signature $(r_1,r_2)$ and degree
\[
d=[K:\mathbb Q]=r_1+2r_2.
\]
Let $D_K$ and $R_K$ denote the discriminant and regulator of $K$, with the convention $R_K=1$ when the unit rank $r_1+r_2-1$ is zero. For every place $v$, let $n_v$ be the local degree: $n_v=[K_v:\mathbb Q_p]$ if $v\mid p$, $n_v=1$ if $v$ is real, and $n_v=2$ if $v$ is complex. For $v\mid p$, let $|\cdot|_v$ be the extension to $K_v$ of the $p$-adic absolute value normalized by $|p|_p=p^{-1}$; at the archimedean places we use the usual absolute value. With these conventions,
\[
\prod_v |a|_v^{n_v}=1
\qquad(a\in K^\times).
\]

Let $E$ be a rigid adelic space over $K$ of dimension $n\ge2$ in the sense of Gaudron \cite{Gaudron}*{Definition~2}. Write $\|\cdot\|_{E,v}$ for its local norms. For $0\ne x\in E$, define
\begin{equation}
H_E(x)=\prod_v\|x\|_{E,v}^{n_v/d}.
\end{equation}
The product formula gives $H_E(ax)=H_E(x)$ for $a\in K^\times$, so this defines a height on $\mathbb P(E)(K)$. For $P=[x]$, write $H_E(P)=H_E(x)$.

For $1\le i\le n$, let $\Lambda_i(E)$ be the Roy--Thunder successive minima \cite{RoyThunder}:
\begin{equation}
\Lambda_i(E)=\inf\left\{\max_{1\le j\le i}H_E(x_j):x_1,\ldots,x_i\in E\text{ are $K$-linearly independent}\right\}.
\end{equation}
For $R>0$, put
\begin{equation}
\mathcal B_E(R)=\{P\in\mathbb P(E)(K):H_E(P)\le R\}.
\end{equation}
After choosing a $K$-basis of $E$, the height $H_E$ is equivalent to the standard projective Weil height, since the local norms differ from the standard ones at only finitely many places. Hence this set is finite by Northcott's theorem \cite{BombieriGubler}*{Theorem~2.4.9}.

Define
\begin{equation}
g_E(R)=\min\left\{N:\mathcal B_E(R)\subset\bigcup_{j=1}^N\mathbb P(W_j)(K),\quad
0\ne W_j\subsetneq E\text{ a $K$-subspace}\right\},
\end{equation}
and
\begin{equation}
h_E(R)=\max\left\{|S|:S\subset\mathcal B_E(R),\ \text{every $n$ distinct points of $S$ span $E$}\right\}.
\end{equation}
Thus $h_E(R)$ is the largest cardinality of a subset of $\mathcal B_E(R)$ in linear general position. Since a proper $K$-subspace contains at most $n-1$ points of such a set,
\begin{equation}\label{eq:intro-hg}
h_E(R)\le(n-1)g_E(R).
\end{equation}

For $1\le m<n$, set
\begin{equation}
\Psi_{E,m}(R)=\left(\frac{R^{n-m+1}}{\Lambda_m(E)\cdots\Lambda_n(E)}\right)^{1/(n-m)},
\end{equation}
and
\begin{equation}
\Psi_E(R)=\min_{1\le m<n}\Psi_{E,m}(R).
\end{equation}
For a rigid adelic space $F$ over $K$, we use the analogous notation $H_F$, $\mathcal B_F(R)$, $g_F(R)$, and $\Lambda_i(F)$. If $\dim_KF=n$, we also write $\Psi_{F,m}(R)$ and $\Psi_F(R)$.

\begin{theorem}\label{thm:main}
There exist effectively computable constants $c_1,C_1>0$ and $B_1\ge2$, depending only on $n,d,r_2,|D_K|$, and $R_K$. For every $n$-dimensional rigid adelic space $E$ over $K$ and every $R\ge B_1\Lambda_n(E)$,
\begin{equation}\label{eq:main}
c_1\Psi_E(R)^d
\le h_E(R)
\le(n-1)g_E(R)
\le C_1\Psi_E(R)^d.
\end{equation}
\end{theorem}

For the standard rigid adelic space $K^n$ with its standard local norms, one has
\[
\Lambda_1(K^n)=\cdots=\Lambda_n(K^n)=1,
\qquad
\Psi_{K^n}(R)=R^{n/(n-1)}\quad(R\ge1).
\]
Hence Theorem~\ref{thm:main} gives the following special case.

\begin{corollary}\label{cor:standard}
There exist effectively computable constants $c,C>0$ and $R_0\ge2$, depending only on $n,d,r_2,|D_K|$, and $R_K$, such that for every $R\ge R_0$,
\[
cR^{dn/(n-1)}
\le h_{K^n}(R)
\le(n-1)g_{K^n}(R)
\le CR^{dn/(n-1)}.
\]
\end{corollary}

The quantity $\Psi_E(R)$ is a minimum over the successive minima. Thus the optimal scale is not determined only by the number of points in the height ball or by the last minimum: different ranges of the Roy--Thunder minima can control the covering problem.

\subsection{Strategy of the proof}

After rescaling the archimedean norms, it is enough to work with the height ball of radius $1$.  The two bounds have a common combinatorial core, but the pigeonhole principle is used in opposite ways.

For the upper bound, the basic counting principle is the same as in the classical pigeonhole proof of Siegel's lemma.  We choose several linear forms in the dual adelic space whose heights are controlled by the first dual Roy--Thunder minima.  After balancing the local norms of these forms and of a point in the height ball, their values are algebraic integers with different archimedean size bounds.  We then let the coefficients of a linear combination vary in boxes whose side lengths are reciprocal to these bounds.  The coefficient tuples range over a product of $k$ boxes in $\mathcal O_K$, whereas all resulting sums lie in a single archimedean box in $\mathcal O_K$.  Since a $K$-box of radius $T$ contains on the order of $T^d$ algebraic integers, with $d=[K:\mathbb Q]$, the source has $k$ such growth factors while the image has only one.  For $k\ge2$, a collision is therefore forced at the scale determined by the product of the dual minima.  The difference of two coefficient tuples gives a nonzero integral relation.  This is the Siegel-lemma-type step: a counting imbalance produces a small vector in a kernel.

There is one additional requirement specific to the covering problem.  The relation cannot be chosen separately for each point; its coefficients must belong to a finite set fixed in advance, so that the kernels of the corresponding linear forms give one finite cover of the entire height ball.  The weighted-relation lemma provides this uniform form of the pigeonhole argument.  Gaudron's transference theorem then replaces the product of the relevant dual minima by the corresponding product of the minima of $E$, which produces the upper-bound scale $\Psi_E(R)^d$.

The lower bound reverses the same counting idea.  An adapted basis turns the successive minima into coordinate ranges in which a vector is guaranteed to have bounded height.  After reduction modulo a prime ideal $\mathfrak p$, these ranges give a product set in $(\mathcal O_K/\mathfrak p)^n$.  If this product set has more than $|\mathcal O_K/\mathfrak p|^{n-1}$ elements, then for every one-dimensional subspace $\ell$ of the residue space, two admissible vectors have the same image in the quotient by $\ell$; their nonzero difference therefore lies in $\ell$.  Thus every projective direction over the residue field has a bounded-height lift.  We take the directions on a rational normal curve.  Any $n$ distinct points on this curve are linearly independent over the residue field, and hence their lifts are linearly independent over $K$.  Choosing $\Nm\mathfrak p$ on the scale $\Psi_E(R)^d$ gives the required number of lifted points.  This also explains why the same product of successive minima appears in both bounds: it is the threshold for forcing a relation in the upper bound and, in the reverse direction, for lifting sufficiently many residue directions in the lower bound.

The arithmetic steps needed to pass between height bounds and these finite counting problems are stated before they are used below.  To keep the proofs of the two main bounds focused on the geometric argument, we postpone the proofs of these auxiliary lemmas to Subsections~\ref{subsec:balanced-proof}--\ref{subsec:prime-interval-proof}.  They account for the ideal-class and unit-group effects in choosing integral representatives, for counting algebraic integers in archimedean boxes, and for choosing the prime ideal at the required scale.

\section{Proof of the Main Theorem}
\subsection{Normalization and reductions}
For $R>0$, let $E_R$ be obtained from $E$ by multiplying every
archimedean norm by $R^{-1}$ and leaving the finite-place norms unchanged.
Then
\begin{equation}
H_{E_R}(x)=R^{-1}H_E(x),\qquad
\Lambda_i(E_R)=R^{-1}\Lambda_i(E)\quad(1\le i\le n),
\end{equation}
and
\begin{equation}
\mathcal{B}_E(R)=\mathcal{B}_{E_R}(1),\qquad
h_E(R)=h_{E_R}(1),\qquad
g_E(R)=g_{E_R}(1),\qquad
\Psi_{E_R}(1)=\Psi_E(R).
\end{equation}
Replacing $E$ by $E_R$, the condition $R\ge B_1\Lambda_n(E)$ becomes $\Lambda_n(E_R)\le B_1^{-1}$. Thus it is enough to work with the unit height ball $\mathcal B_{E_R}(1)$.

Let $\cO_K$ be the ring of integers of $K$.  For every nonzero ideal
$\mathfrak a\subset\cO_K$, write
$\Nm\mathfrak a=|\cO_K/\mathfrak a|$ for its absolute norm; we use
$\Nm_{K/\mathbb Q}$ for the field norm.

\begin{lemma}\label{lem:representatives}
There exists a constant $C_2\geq 1$, effectively computable from $d, r_2,|D_K|, R_K$ with the following property. For every rigid adelic space $F$ over $K$ and every $0\neq x\in F$ there is $a\in K^\times$ such that
\begin{equation}
    \|ax\|_{F,v}\le
\begin{cases}
1,& v\nmid\infty,\\
C_2H_F(x),& v\mid\infty.
\end{cases}
\end{equation}
\end{lemma}
We postpone the proof of Lemma~\ref{lem:representatives} to Subsection~\ref{subsec:balanced-proof}.

For $a\in K$, write
\[
\|a\|_\infty=\max_{\sigma:K\hookrightarrow\mathbb C}|\sigma(a)|,
\]
and for $T\ge0$ put
\[
N_K(T)=\#\{a\in\cO_K:\|a\|_\infty\le T\}.
\]

\begin{lemma}\label{lem:box-count}
There are constants $c_4>0$ and $C_4\ge1$, effectively computable from $d,r_2, |D_K|$, such that for every $T\ge0$:
\begin{enumerate}
\item[(i)]
\[
c_4\,(1+T^d)\le N_K(T)\le C_4\,(1+T^d);
\]
\item[(ii)]
\begin{equation}\label{eq:box-doubling}
N_K(2T)\le C_4N_K(T).
\end{equation}
\end{enumerate}
\end{lemma}
We postpone the proof of Lemma~\ref{lem:box-count} to Subsection~\ref{subsec:box-count-proof}.

\begin{lemma}\label{lem:adapted-basis}
Let $F$ be an $s$-dimensional rigid adelic space over $K$.  There is a
$K$-basis $z_1,\ldots,z_s$ of $F$ such that
\[
 H_F(z_i)\le2\Lambda_i(F)
 \qquad(1\le i\le s).
\]
\end{lemma}

\begin{proof}
For each $1\le i\le s$, the definition of $\Lambda_i(F)$ gives at least
$i$ linearly independent vectors of height at most $2\Lambda_i(F)$.  Having
chosen $z_1,\ldots,z_{i-1}$, choose $z_i$ among these vectors outside
$\operatorname{span}_K(z_1,\ldots,z_{i-1})$.  Then
$H_F(z_i)\le2\Lambda_i(F)$, and $z_1,\ldots,z_s$ is a basis of $F$.
\end{proof}

\subsection{Upper bound}

\begin{lemma}\label{lem:weighted-relation}
Let $k\ge2$ and $M>0$.  There is a constant $C_5>0$, effectively computable from
$M,k,d,r_2,|D_K|$, such that for every $\nu_1,\ldots,\nu_k>0$ there is a
finite set $\mathcal D\subset\cO_K^k\setminus\{0\}$ with the following properties.
\begin{enumerate}
\item[(i)]
\[
|\mathcal D|
\le
C_5
\left[
1+(\nu_1\cdots\nu_k)^{d/(k-1)}
\right].
\]

\item[(ii)] If $y_1,\ldots,y_k\in\cO_K$ satisfy
\[
\|y_i\|_\infty\le M\nu_i
\qquad(1\le i\le k),
\]
then some $(\delta_1,\ldots,\delta_k)\in\mathcal D$ satisfies
\[
\delta_1y_1+\cdots+\delta_ky_k=0.
\]
\end{enumerate}
\end{lemma}

We postpone the proof of Lemma~\ref{lem:weighted-relation} to Subsection~\ref{subsec:weighted-proof}.

\begin{proposition}\label{prop:upper}
There is a constant $C_6>0$, effectively computable from
$n,d,r_2,|D_K|$, and $R_K$, such that every $n$-dimensional rigid adelic space $F$ over $K$ satisfies:
\begin{enumerate}
\item[(i)]
\[
g_F(1)
\le C_6\left[1+\Psi_F(1)^d\right].
\]
\item[(ii)] If $\Lambda_n(F)\le1$, then
\[
g_F(1)
\le C_6\Psi_F(1)^d.
\]
\end{enumerate}
\end{proposition}

\begin{proof}
Fix $1\le m<n$ and set $k=n-m+1$.  Endow
$F^\vee=\operatorname{Hom}_K(F,K)$ with the local operator norms
\[
\|\varphi\|_{F^\vee,v}
=
\sup_{0\ne x\in F\otimes_K K_v}
\frac{|\varphi(x)|_v}{\|x\|_{F,v}}.
\]
This dual adelic structure is rigid.
Lemma~\ref{lem:adapted-basis} gives a $K$-basis
$\varphi_1,\ldots,\varphi_n$ of $F^\vee$ such that
\[
H_{F^\vee}(\varphi_i)
\le
2\Lambda_i(F^\vee)
\qquad(1\le i\le n).
\]
For $1\le i\le k$, rescale $\varphi_i$ using
Lemma~\ref{lem:representatives}.  With
\[
\nu_i=2C_2\Lambda_i(F^\vee),
\]
we may assume
\[
\|\varphi_i\|_{F^\vee,v}\le
\begin{cases}
1,&v\nmid\infty,\\
\nu_i,&v\mid\infty.
\end{cases}
\]
Lemma~\ref{lem:weighted-relation}, applied with $M=C_2$ and
$\nu_1,\ldots,\nu_k$, gives a finite set
$\mathcal D\subset\cO_K^k\setminus\{0\}$ such that
\[
|\mathcal D|
\le
C_5
\left[
1+(\nu_1\cdots\nu_k)^{d/(k-1)}
\right].
\]

Let $P\in\mathcal{B}_F(1)$ and choose a representative $0\ne x\in F$.
Since $H_F(x)=H_F(P)\le1$, Lemma~\ref{lem:representatives} allows us
to rescale $x$, without changing $P$, so that
\[
\|x\|_{F,v}\le
\begin{cases}
1,&v\nmid\infty,\\
C_2,&v\mid\infty.
\end{cases}
\]
Set $y_i=\varphi_i(x)$ for $1\le i\le k$.  By the definition of the
operator norm,
\[
|y_i|_v
\le
\|\varphi_i\|_{F^\vee,v}\|x\|_{F,v}
\le
\begin{cases}
1,&v\nmid\infty,\\
C_2\nu_i,&v\mid\infty.
\end{cases}
\]
Thus $y_i$ is integral at every finite place and hence belongs to
$\cO_K$, while the archimedean bounds give
$\|y_i\|_\infty\le C_2\nu_i$.  Hence
Lemma~\ref{lem:weighted-relation} gives
$\boldsymbol\delta=(\delta_1,\ldots,\delta_k)\in\mathcal D$ such that
\[
\delta_1\varphi_1(x)+\cdots+\delta_k\varphi_k(x)=0.
\]
Define
\[
\varphi_{\boldsymbol\delta}=\sum_{i=1}^k \delta_i\varphi_i\in F^\vee.
\]
Since $\boldsymbol\delta\ne0$ and the $\varphi_i$ are $K$-linearly independent,
$\varphi_{\boldsymbol\delta}\ne0$ and $\varphi_{\boldsymbol\delta}(x)=0$.  Since
$\dim_K F=n\ge2$, the kernel of $\varphi_{\boldsymbol\delta}$ is a nonzero
proper $K$-subspace of $F$, and
\[
P=[x]\in\mathbb P(\ker\varphi_{\boldsymbol\delta})(K).
\]
Therefore the projective hyperplanes
\[
\left\{
\mathbb P(\ker\varphi_{\boldsymbol\delta})(K):\boldsymbol\delta\in\mathcal D
\right\}
\]
cover $\mathcal{B}_F(1)$, and
\[
g_F(1)
\le
C_5
\left[
1+(\nu_1\cdots\nu_k)^{d/(k-1)}
\right].
\]

Gaudron's transference theorem \cite{Gaudron}*{p.~66, Theorem~36} gives
\[
\Lambda_i(F^\vee)
\le
\frac{n|D_K|^{1/d}}
{\Lambda_{n-i+1}(F)}
\qquad(1\le i\le k).
\]
Multiplying these inequalities and using $k=n-m+1$ gives
\[
\prod_{i=1}^k\Lambda_i(F^\vee)
\le
\frac{(n|D_K|^{1/d})^k}
{\Lambda_m(F)\cdots\Lambda_n(F)}.
\]
Hence
\[
(\nu_1\cdots\nu_k)^{d/(k-1)}
\le
\bigl(2C_2n|D_K|^{1/d}\bigr)^{kd/(k-1)}
\Psi_{F,m}(1)^d.
\]
The factor
$\bigl(2C_2n|D_K|^{1/d}\bigr)^{kd/(k-1)}$ is bounded uniformly for
$m\in\{1,\ldots,n-1\}$.  Hence we may choose $C_6$, depending only on
$n,d,r_2,|D_K|$, and $R_K$, such that
\[
g_F(1)
\le
\frac{C_6}{2}\left[1+\Psi_{F,m}(1)^d\right]
\qquad(1\le m<n).
\]
Taking the minimum over $m$ gives
\[
g_F(1)
\le
\frac{C_6}{2}\left[1+\Psi_F(1)^d\right].
\]
This proves part~(i).  If $\Lambda_n(F)\le1$, then $\Lambda_i(F)\le1$ for every $i$, and
\[
\Psi_{F,m}(1)^d
=
\bigl(\Lambda_m(F)\cdots\Lambda_n(F)\bigr)^{-d/(n-m)}
\ge1
\qquad(1\le m<n).
\]
Thus $\Psi_F(1)^d\ge1$, and
\[
g_F(1)
\le C_6\Psi_F(1)^d.
\]
\end{proof}
\subsection{Lower bound}

\begin{lemma}
\label{lem:residue}
There is a constant $c_7>0$, effectively computable from $d, r_2, |D_K|$, with the following property. For every nonzero prime ideal $\mathfrak p\subset\cO_K$, with
$q=\Nm\mathfrak p$, and every $R\ge1$, there is a subset
$S(R,\mathfrak p)\subset\cO_K/\mathfrak p$ such that:
\begin{enumerate}
\item[(i)]
\[
|S(R,\mathfrak p)|\ge c_7\min\{q,R^d\};
\]
\item[(ii)] every class in $S(R,\mathfrak p)$ has a representative
$a\in\cO_K$ with $\|a\|_\infty\le R/2$.
\end{enumerate}
\end{lemma}

We postpone the proof of Lemma~\ref{lem:residue} to Subsection~\ref{subsec:residue-product-proofs}.

\begin{lemma}\label{lem:product}
Let $C_2$ be the constant in Lemma~\ref{lem:representatives}.  There are
constants $0<c_8<1$ and $C_8\ge2$, effectively computable from $n, d, r_2, |D_K|, R_K$, with the following property.  Let $F$ be
an $n$-dimensional rigid adelic space over $K$, and write
\[
\lambda_i=\Lambda_i(F),
\qquad
R_i=\frac{1}{8nC_2\lambda_i}
\qquad(1\le i\le n).
\]
If $\mathfrak p\subset\cO_K$ is a nonzero prime ideal, $q=\Nm\mathfrak p$, and
\[
\lambda_n\le c_8,
\qquad
C_8\le q\le c_8\Psi_F(1)^d,
\]
then:
\begin{enumerate}
\item[(i)] $R_i\ge1$ for every $1\le i\le n$;
\item[(ii)] if, for each $1\le i\le n$, the set $S_i=S(R_i,\mathfrak p)$ is chosen as in Lemma~\ref{lem:residue}, then
\[
\prod_{i=1}^n|S_i|>q^{n-1}.
\]
\end{enumerate}
\end{lemma}

We postpone the proof of Lemma~\ref{lem:product} to Subsection~\ref{subsec:residue-product-proofs}.

\begin{lemma}\label{lem:prime-interval}
There is a constant $C_9>1$, effectively computable from $d$ and $|D_K|$, such that for every $X\ge2$ there is a nonzero prime ideal
$\mathfrak p\subset\cO_K$ with
\[
X\le\Nm\mathfrak p\le C_9X.
\]
\end{lemma}
We postpone the proof of Lemma~\ref{lem:prime-interval} to Subsection~\ref{subsec:prime-interval-proof}.

\begin{proposition}\label{prop:lower}
There are constants $c_1>0$ and $B_1\ge2$, effectively computable from $n, d, r_2, |D_K|, R_K$ with the following property. For every $n$-dimensional rigid adelic space $E$ over $K$ and every $R\ge B_1\Lambda_n(E)$,
\[
 h_E(R)\ge c_1\Psi_E(R)^d.
\]
\end{proposition}

\begin{proof}
Let $c_8,C_8$ be as in Lemma~\ref{lem:product}, and let $C_9$ be the
constant from Lemma~\ref{lem:prime-interval}.  Write
\[
\lambda_i=\Lambda_i(E_R)=\frac{\Lambda_i(E)}{R}
\qquad(1\le i\le n).
\]
Set
\[
B_1
=
\max\left\{
2,\,
c_8^{-1},\,
\left(
\frac{C_9}{c_8}\max\{2,C_8,n-1\}
\right)^{(n-1)/(dn)}
\right\}.
\]
Since $R\ge B_1\Lambda_n(E)$,
\[
\lambda_n
=
\frac{\Lambda_n(E)}{R}
\le
\frac1{B_1}
\le c_8.
\]
For every $1\le m<n$,
\[
\lambda_m\cdots\lambda_n
\le
\lambda_n^{\,n-m+1},
\]
and hence
\[
(\lambda_m\cdots\lambda_n)^{-d/(n-m)}
\ge
\lambda_n^{-d(n-m+1)/(n-m)}
\ge
\lambda_n^{-d n/(n-1)}.
\]
Taking the minimum over $m$ gives
\[
\Psi_E(R)^d
=
\Psi_{E_R}(1)^d
\ge
\lambda_n^{-d n/(n-1)}
\ge
B_1^{dn/(n-1)}
\ge
\frac{C_9}{c_8}\max\{2,C_8,n-1\}.
\]
Therefore
\[
\frac{c_8}{C_9}\Psi_E(R)^d\ge\max\{2,C_8,n-1\}.
\]
Applying Lemma~\ref{lem:prime-interval} with
$X=(c_8/C_9)\Psi_E(R)^d$ gives a nonzero prime ideal
$\mathfrak p\subset\cO_K$, with $q=\Nm\mathfrak p$, such that
\[
\max\{2,C_8,n-1\}
\le
\frac{c_8}{C_9}\Psi_E(R)^d
\le q
\le c_8\Psi_E(R)^d.
\]
Applying Lemma~\ref{lem:adapted-basis} to $E_R$ and then
Lemma~\ref{lem:representatives} to the resulting basis vectors, we obtain a
$K$-basis $u_1,\ldots,u_n$ of $E_R$ satisfying
\[
\|u_i\|_{E_R,v}\le
\begin{cases}
1,&v\nmid\infty,\\
2C_2\lambda_i,&v\mid\infty.
\end{cases}
\]
All reductions modulo $\mathfrak p$ below are taken in the coordinate
space $\cO_K^n$ relative to this basis.  Put
\[
R_i=
\frac{1}{8nC_2\lambda_i}
\qquad(1\le i\le n).
\]
Lemma~\ref{lem:product} gives $R_i\ge1$ for all $i$; if $S_i=S(R_i,\mathfrak p)$ for $1\le i\le n$, then
\[
|S_1\times\cdots\times S_n|
=
\prod_{i=1}^n|S_i|
>
q^{n-1}.
\]

Let
$\ell\subset(\cO_K/\mathfrak p)^n$ be a one-dimensional subspace.
Since
\[
\left|(\cO_K/\mathfrak p)^n/\ell\right|
=q^{n-1}
<|S_1\times\cdots\times S_n|,
\]
the quotient map from $S_1\times\cdots\times S_n$ to
$(\cO_K/\mathfrak p)^n/\ell$ is not injective.  Hence there are
distinct $\bar a,\bar b\in S_1\times\cdots\times S_n$ with
$0\ne\bar a-\bar b\in\ell$.  Write
$\bar a=(\bar a_1,\ldots,\bar a_n)$ and
$\bar b=(\bar b_1,\ldots,\bar b_n)$, and choose representatives
$a_i,b_i\in\cO_K$ as in Lemma~\ref{lem:residue},
so that
\[
\|a_i\|_\infty,\|b_i\|_\infty\le\frac{R_i}{2},
\qquad
\|a_i-b_i\|_\infty\le R_i.
\]
Put
\[
x_\ell
=
\sum_{i=1}^n(a_i-b_i)u_i.
\]
Its coefficient vector with respect to the basis
$u_1,\ldots,u_n$, reduced modulo $\mathfrak p$, is the nonzero vector
$\bar a-\bar b\in\ell$.  Hence $x_\ell\ne0$.

For $v\nmid\infty$, we have $a_i-b_i\in\cO_K$ and hence
$|a_i-b_i|_v\le1$.  Since $\|\cdot\|_{E_R,v}$ is ultrametric,
\[
\|x_\ell\|_{E_R,v}
\le
\max_{1\le i\le n}
|a_i-b_i|_v\|u_i\|_{E_R,v}
\le1.
\]
For $v\mid\infty$,
\[
\begin{aligned}
\|x_\ell\|_{E_R,v}
&\le
\sum_{i=1}^n
|a_i-b_i|_v\|u_i\|_{E_R,v}\\
&\le
\sum_{i=1}^n
2R_iC_2\lambda_i
=
\frac14.
\end{aligned}
\]
Hence $H_{E_R}(x_\ell)\le1$, so
$[x_\ell]\in\mathcal{B}_{E_R}(1)$.

Consider the rational normal curve used in \cite{BHPT}*{Theorem~1}:
\[
\mathcal R_{\mathfrak p}
=
\{[1:t:t^2:\cdots:t^{n-1}]:t\in\cO_K/\mathfrak p\}
\cup
\{[0:\cdots:0:1]\}
\subset\mathbb P^{n-1}(\cO_K/\mathfrak p).
\]
The $q=\Nm\mathfrak p$ values of $t$ give distinct points, none equal to $[0:\cdots:0:1]$, so $\mathcal R_{\mathfrak p}$ has $q+1$ points.
For distinct parameters
$t_1,\ldots,t_n$, the matrix whose columns are the corresponding vectors
$(1,t_i,\ldots,t_i^{n-1})^{\mathsf T}$ has determinant
\[
\prod_{1\le i<j\le n}(t_j-t_i)\ne0.
\]
If $[0:\cdots:0:1]$ occurs, expansion along the corresponding column gives
a $(n-1)\times(n-1)$ Vandermonde determinant.  Thus any nonzero
representatives of $n$ distinct points of $\mathcal R_{\mathfrak p}$ are
linearly independent over $\cO_K/\mathfrak p$.

For each $\ell\in\mathcal R_{\mathfrak p}$, choose the vector
$x_\ell$ constructed above. If $\ell_1,\ldots,\ell_n$ are distinct, the
coordinate matrix of $x_{\ell_1},\ldots,x_{\ell_n}$ with respect to
$u_1,\ldots,u_n$ has entries in $\cO_K$, because each $x_\ell$ is an
$\cO_K$-linear combination of this basis.  Its reduction modulo
$\mathfrak p$ has as columns nonzero representatives of
$\ell_1,\ldots,\ell_n$, which are $n$ distinct points of
$\mathcal R_{\mathfrak p}$.  Its determinant is therefore nonzero modulo
$\mathfrak p$, and hence $x_{\ell_1},\ldots,x_{\ell_n}$ are
$K$-linearly independent.

Since $q\ge n-1$, the set $\mathcal R_{\mathfrak p}$ has at least $n$ points. The points $[x_\ell]$ are distinct: if $[x_\ell]=[x_{\ell'}]$ for distinct $\ell,\ell'\in\mathcal R_{\mathfrak p}$, extend $\{\ell,\ell'\}$ to $n$ distinct points of $\mathcal R_{\mathfrak p}$. The corresponding vectors would then be linearly dependent, contradicting the fact proved above that the vectors attached to any $n$ distinct points of $\mathcal R_{\mathfrak p}$ are linearly independent. Thus the $q+1$ points
\[
[x_\ell]\in\mathcal{B}_{E_R}(1)
\qquad(\ell\in\mathcal R_{\mathfrak p})
\]
have the property that any $n$ of them span $E_R$.  With $c_1=c_8/C_9$,
\[
h_E(R)
=
h_{E_R}(1)
\ge q+1
\ge q
\ge
c_1\Psi_E(R)^d.
\]
\end{proof}

\begin{proof}[Proof of Theorem~\ref{thm:main}]
Let $c_1$ and $B_1$ be as in Proposition~\ref{prop:lower}.  If
$R\ge B_1\Lambda_n(E)$, then $B_1\ge2$ gives
\[
 \Lambda_n(E_R)=\frac{\Lambda_n(E)}R\le1.
\]
Proposition~\ref{prop:upper}, Proposition~\ref{prop:lower}, and
\eqref{eq:intro-hg} therefore give
\[
c_1\Psi_E(R)^d
\le h_E(R)
\le (n-1)g_E(R)
\le (n-1)C_6\Psi_E(R)^d.
\]
Thus \eqref{eq:main} holds with $C_1=(n-1)C_6$.
\end{proof}

\subsection{Proof of the balanced representative lemma}\label{subsec:balanced-proof}
\begin{proof}[Proof of Lemma~\ref{lem:representatives}]
For a finite place $v$, let $\cO_v$ denote the valuation ring of $K_v$. For $0\ne x\in F$, the homogeneity and ultrametricity of the local norm give a fractional $\cO_v$-ideal
\[
\{c\in K_v:\|cx\|_{F,v}\le1\}=\xi_v\cO_v
\]
for some $\xi_v\in K_v^\times$ satisfying $|\xi_v|_v\|x\|_{F,v}=1$. By the adelic condition, this ideal is $\cO_v$ for all but finitely many finite places. Hence
\[
\mathfrak a_x=\{c\in K:\|cx\|_{F,v}\le1\text{ for every }v\nmid\infty\}
\]
is a nonzero fractional ideal. At each finite place, we have
\[
\mathfrak a_x\cO_v=\xi_v\cO_v.
\]
By Minkowski's theorem (see \cite{Neukirch}*{Chapter~I, Section~6, Exercise~3}), we can choose an integral ideal $\mathfrak c$ in the class $[\mathfrak a_x]^{-1}$ such that
\begin{equation}\label{eq:class-rep-norm-bound}
1\leq \Nm\mathfrak c\le
\left(\frac4\pi\right)^{r_2}
\frac{d!}{d^d}|D_K|^{1/2}.
\end{equation}
Write $\mathfrak a_x\mathfrak c=(a_0)$.  Since
$\mathfrak c\subset\cO_K$, one has $a_0\in\mathfrak a_x$.
For $v\nmid\infty$, the equality
$a_0\cO_v=\xi_v\mathfrak c\cO_v$ shows that $a_0/\xi_v$ is a local generator of $\mathfrak c\cO_v$. Hence
\[
\|a_0x\|_{F,v}=|a_0|_v\|x\|_{F,v}=|a_0/\xi_v|_v\le1.
\]
Taking the product over the finite places and using the ideal norm gives
\begin{equation}
\prod_{v\nmid\infty}\|a_0x\|_{F,v}^{n_v/d}=(\Nm\mathfrak c)^{-1/d}.
\end{equation}
Thus the archimedean contribution is
\begin{equation}
Q:=\prod_{v\mid\infty}\|a_0x\|_{F,v}^{n_v/d}
=H_F(x)(\Nm\mathfrak c)^{1/d}.
\end{equation}
Applying \eqref{eq:class-rep-norm-bound}, we deduce that
\begin{equation}
Q\le\left(\left(\frac4\pi\right)^{r_2}\frac{d!}{d^d}|D_K|^{1/2}\right)^{1/d}H_F(x). 
\end{equation}
By Dirichlet's theorem, $r=r_1+r_2-1$ is the unit rank of $K$.  If $r=0$, take
$\varepsilon=1$ and put $C_{\mathrm u}=0$.
Otherwise, let $H$ denote the absolute multiplicative height. Apply Bugeaud and Gy\H{o}ry \cite{BugeaudGyory}*{pp.~70--71, Lemma~1(ii)} with $S$ equal to the set of archimedean places. Their constant involves the lower height bound $\delta_K$; using the explicit degree-dependent choice of $\delta_K$ given immediately before Lemma~1, we obtain an effectively computable constant $C_{\mathrm{BG}}>0$, depending only on $d$ and $r_2$, and fundamental units
$\varepsilon_1,\ldots,\varepsilon_r$ such that
\[
\log H(\varepsilon_j)\le C_{\mathrm{BG}}R_K
\qquad(1\le j\le r).
\]
Put
\[
\ell_K(\varepsilon)=
(n_v\log|\varepsilon|_v)_{v\mid\infty}.
\]
Since $\varepsilon$ is a unit, the product formula gives
\[
\prod_{v\mid\infty}|\varepsilon|_v^{n_v}=1.
\]
Hence
\begin{equation}
\sum_{v\mid\infty}n_v\bigl|\log|\varepsilon|_v\bigr|
=
2\sum_{v\mid\infty}n_v\max\{\log|\varepsilon|_v,0\}
=2d\log H(\varepsilon).
\end{equation}
Therefore
\begin{equation}
\max_{v\mid\infty}\bigl|\bigl(\ell_K(\varepsilon_j)\bigr)_v\bigr|
\le 2d C_{\mathrm{BG}}R_K.
\end{equation}
Set
\[
y=
\bigl(n_v(\log\|a_0x\|_{F,v}-\log Q)\bigr)_{v\mid\infty}.
\]
Then $y$ lies in the hyperplane $\sum_v y_v=0$, and
$\ell_K(\varepsilon_1),\ldots,\ell_K(\varepsilon_r)$ form a real basis
of this hyperplane.  Write
$y=\sum_{j=1}^r t_j\ell_K(\varepsilon_j)$, choose $m_j\in\mathbb Z$ with
$|t_j+m_j|\le1/2$, and put
\[
\varepsilon=\prod_{j=1}^{r}\varepsilon_j^{m_j}.
\]
Then
\[
\max_{v\mid\infty}\left|y_v+\bigl(\ell_K(\varepsilon)\bigr)_v\right|
\leq d r C_{\mathrm{BG}}R_K
=:C_{\mathrm u}.
\]
When $r>0$, for every $v\mid\infty$,
\[
 n_v\log\frac{\|\varepsilon a_0x\|_{F,v}}{Q}
 =
 y_v+\bigl(\ell_K(\varepsilon)\bigr)_v,
\]
and therefore
\[
 \|\varepsilon a_0x\|_{F,v}
 \le e^{C_{\mathrm u}/n_v}Q\le e^{C_{\mathrm u}}Q.
\]
When $r=0$, we have $\varepsilon=1$, $C_{\mathrm u}=0$, and $Q=\|a_0x\|_{F,v}$. Thus in all cases
\[
\|\varepsilon a_0x\|_{F,v}\le e^{C_{\mathrm u}}Q
\qquad(v\mid\infty).
\]
Since $|\varepsilon|_v=1$ at every finite place, $a=\varepsilon a_0$
satisfies the finite-place bound in Lemma~\ref{lem:representatives}.  Using the bound for $Q$ above, the archimedean estimate holds with
\[
C_2
=
e^{C_{\mathrm u}}
\left(
\left(\frac4\pi\right)^{r_2}
\frac{d!}{d^d}|D_K|^{1/2}
\right)^{1/d}.
\]
\end{proof}

\subsection{Lattice and box-count estimates}\label{subsec:box-count-proof}
Set
\[
K_\infty
=
K\otimes_{\mathbb Q}\mathbb R
\simeq
\prod_{v\mid\infty}K_v
\simeq
\mathbb R^{r_1}\times\mathbb C^{r_2}.
\]
Let $\iota:K\hookrightarrow K_\infty$ be the Minkowski embedding, obtained by choosing one embedding representing each archimedean place.  We use the max norm on $K_\infty$:
\[
\|(x_v)_{v\mid\infty}\|_\infty=\max_{v\mid\infty}|x_v|.
\]
We equip $K_\infty$ with the product of the usual Lebesgue measures and write
\[
B_\infty(R)=\{x\in K_\infty:\|x\|_\infty\le R\}
\qquad(R\ge0).
\]
The norm agrees with the notation already used on $K$: $\|\iota(a)\|_\infty=\|a\|_\infty$ for $a\in K$.  Its volume is
\[
\vol B_\infty(R)=2^{r_1}\pi^{r_2}R^d.
\]
For a full lattice $L\subset K_\infty$, denote its covolume by
$\covol(L)$ and its covering radius by
\[
\rho_\infty(L)
=
\sup_{x\in K_\infty}\inf_{\lambda\in L}
\|x-\lambda\|_\infty.
\]
\[
\covol\bigl(\iota(\mathfrak a)\bigr)
=2^{-r_2}|D_K|^{1/2}\Nm\mathfrak a
\]
for every nonzero integral ideal $\mathfrak a\subset\cO_K$; see \cite{Neukirch}*{Chapter~I, Proposition~5.2}.
\begin{lemma}
\label{lem:ideal-covering}
There is a constant $C_3>0$, effectively computable from
$d, r_2, |D_K|$, such that every nonzero integral ideal
$\mathfrak a\subset\cO_K$ satisfies
\begin{equation}
    \rho_\infty\bigl(\iota(\mathfrak a)\bigr)
\le C_3(\Nm\mathfrak a)^{1/d}.
\end{equation}
\end{lemma}
\begin{proof}
Let
\[
\lambda_1^{(\infty)}\le\cdots\le\lambda_d^{(\infty)}
\]
be the successive minima of the lattice $\iota(\mathfrak a)$ with respect
to $B_\infty(1)$.  If $0\ne\alpha\in\mathfrak a$, then
$(\alpha)\subseteq\mathfrak a$, so
\[
|\Nm_{K/\mathbb Q}(\alpha)|\ge \Nm\mathfrak a.
\]
On the other hand,
\[
|\Nm_{K/\mathbb Q}(\alpha)|
=\prod_{v\mid\infty}|\alpha|_v^{n_v}
\le\|\iota(\alpha)\|_\infty^d.
\]
Hence
\[
\lambda_1^{(\infty)}\ge(\Nm\mathfrak a)^{1/d}.
\]

By Minkowski's theorem \cite{Cassels}*{Chapter~VIII, \S4.3, Theorem~V},
\[
\lambda_1^{(\infty)}\cdots\lambda_d^{(\infty)}
\le
\frac{2^d\covol(\iota(\mathfrak a))}{\vol B_\infty(1)}
=
\left(\frac{2}{\pi}\right)^{r_2}|D_K|^{1/2}\Nm\mathfrak a.
\]
Since $\lambda_i^{(\infty)}\ge\lambda_1^{(\infty)}$ for every $i$,
\[
\lambda_d^{(\infty)}
\bigl(\lambda_1^{(\infty)}\bigr)^{d-1}
\le
\lambda_1^{(\infty)}\cdots\lambda_d^{(\infty)}.
\]
Together with the lower bound for $\lambda_1^{(\infty)}$, this gives
\[
\lambda_d^{(\infty)}
\le
\left(\frac{2}{\pi}\right)^{r_2}|D_K|^{1/2}
(\Nm\mathfrak a)^{1/d}.
\]

Since $\iota(\mathfrak a)$ is discrete and $B_\infty(1)$ is compact, we can choose linearly independent vectors
$u_1,\ldots,u_d\in\iota(\mathfrak a)$ with
\[
\|u_i\|_\infty\le\lambda_d^{(\infty)}
\qquad(1\le i\le d).
\]
Every point of the centered fundamental parallelotope of the sublattice
$\sum_{i=1}^d\mathbb Z u_i$,
\[
\left\{
\sum_{i=1}^{d} t_i u_i
:
-\frac12\le t_i<\frac12
\text{ for }1\le i\le d
\right\},
\]
has norm at most
\[
\frac12\sum_{i=1}^{d}\|u_i\|_\infty
\le
\frac{d}{2}\lambda_d^{(\infty)}.
\]
Thus this parallelotope is contained in
$B_\infty\bigl(\frac d2\lambda_d^{(\infty)}\bigr)$.
Take
\[
C_3
=
\frac{d}{2}\left(\frac{2}{\pi}\right)^{r_2}|D_K|^{1/2}.
\]
Because $\sum_{i=1}^d\mathbb Z u_i\subset\iota(\mathfrak a)$,
\[
\rho_\infty\bigl(\iota(\mathfrak a)\bigr)
\le
\rho_\infty\left(\sum_{i=1}^d\mathbb Z u_i\right)
\le
\frac{d}{2}\lambda_d^{(\infty)}
\le
C_3(\Nm\mathfrak a)^{1/d}.
\]
\end{proof}
\begin{proof}
Set $L=\iota(\cO_K)$ and $\rho=\rho_\infty(L)$.  Then
\[
\covol(L)=2^{-r_2}|D_K|^{1/2},
\qquad
N_K(R)=\#\bigl(L\cap B_\infty(R)\bigr).
\]
Lemma~\ref{lem:ideal-covering}, applied to $\mathfrak a=\cO_K$, shows that $\rho\le C_3$.

Start with a bounded half-open fundamental parallelotope for $L$. Translate each of its points by the negative of a nearest lattice point, resolving ties by a fixed enumeration of $L$. The tie-breaking makes the selection measurable, and translation by lattice vectors preserves the class modulo $L$; hence the resulting set $\mathcal F$ is a measurable fundamental domain. Moreover,
\[
K_\infty=\bigsqcup_{\lambda\in L}(\lambda+\mathcal F),
\qquad
\vol(\mathcal F)=2^{-r_2}|D_K|^{1/2},
\qquad
\mathcal F\subset B_\infty(\rho).
\]

If $\lambda\in L\cap B_\infty(R)$, then
\[
\lambda+\mathcal F\subset B_\infty(R+\rho).
\]
These translates of $\mathcal F$ are disjoint, and hence
\[
N_K(R)\,2^{-r_2}|D_K|^{1/2}
\le
2^{r_1}\pi^{r_2}(R+\rho)^d.
\]
On the other hand, suppose $R\ge\rho$ and let $x\in B_\infty(R-\rho)$. Write
$x=\lambda+y$ with $\lambda\in L$ and $y\in\mathcal F$.  Then
\[
\|\lambda\|_\infty
\le
\|x\|_\infty+\|y\|_\infty
\le R,
\]
so $\lambda\in L\cap B_\infty(R)$.  Thus
\[
B_\infty(R-\rho)
\subset
\bigcup_{\lambda\in L\cap B_\infty(R)}(\lambda+\mathcal F),
\]
and hence
\[
2^{r_1}\pi^{r_2}(R-\rho)^d
\le
N_K(R)\,2^{-r_2}|D_K|^{1/2}
\qquad(R\ge\rho).
\]

For $R\ge\max\{1,2C_3\}$, we have $\rho\le R/2$, and therefore
\[
N_K(R)
\ge
\frac{2^{r_1}\pi^{r_2}}{2^{-r_2}|D_K|^{1/2}}
\left(\frac R2\right)^d
\ge
\frac{2^{r_1}\pi^{r_2}}
{2^{d+1-r_2}|D_K|^{1/2}}(1+R^d).
\]
For $0\le R<\max\{1,2C_3\}$, the lattice point $0$ gives $N_K(R)\ge1$, hence
\[
N_K(R)
\ge
\frac{1}{1+\max\{1,2C_3\}^d}(1+R^d).
\]
Thus part~(i) holds with
\[
c_4=
\min\left\{
\frac{2^{r_1}\pi^{r_2}}
{2^{d+1-r_2}|D_K|^{1/2}},
\frac{1}{1+\max\{1,2C_3\}^d}
\right\}.
\]

The upper volume estimate and $\rho\le C_3$ give
\[
N_K(R)
\le
\frac{2^{r_1}\pi^{r_2}}{2^{-r_2}|D_K|^{1/2}}
\bigl(R+C_3\bigr)^d.
\]
Since $R+C_3\le(1+C_3)(1+R)$ and
$(1+R)^d\le2^{d-1}(1+R^d)$, we obtain
\[
N_K(R)\le b(1+R^d),
\]
where
\[
b=
\frac{2^{r_1}\pi^{r_2}}{2^{-r_2}|D_K|^{1/2}}\,
2^{d-1}
\bigl(1+C_3\bigr)^d.
\]
Taking
\[
C_4=\max\{b,\,2^d b/c_4,\,1\}
\]
gives
\[
N_K(R)\le C_4(1+R^d)
\]
and
\[
N_K(2R)
\le
b(1+2^d R^d)
\le
2^d b(1+R^d)
\le
\frac{2^d b}{c_4}N_K(R)
\le
C_4N_K(R).
\]
\end{proof}

\subsection{Proof of the weighted-relation lemma}\label{subsec:weighted-proof}
\begin{proof}[Proof of Lemma~\ref{lem:weighted-relation}]

Let $c_4$ and $C_4$ be the constants from Lemma~\ref{lem:box-count}. Set
\[
\mathcal N(T)=\prod_{i=1}^kN_K(T/\nu_i),
\qquad
V=\nu_1\cdots\nu_k,
\qquad
A=2C_4\max\{1,(kM)^d\}.
\]
By the lower bound in Lemma~\ref{lem:box-count},
\[
\mathcal N(T)
\ge
c_4^k\prod_{i=1}^k\left(1+(T/\nu_i)^d\right)
\ge
c_4^k
\frac{T^{d k}}{V^d}.
\]
Hence, for $T\ge1$,
\[
\frac{\mathcal N(T)}{1+T^d}
\ge
\frac{c_4^k}{2V^d}
T^{d(k-1)}
\longrightarrow\infty.
\]

Choose $0<T_0<\min_{1\le i\le k}\nu_i$.  If $0\ne a\in\cO_K$, then
$1\le |\Nm_{K/\mathbb Q}(a)|\le\|a\|_\infty^d$, so $N_K(R)=1$ for
$0\le R<1$.  Thus
\[
\mathcal N(T_0)=1\le A(1+T_0^d).
\]
Since $\mathcal N(T)/(1+T^d)\to\infty$, choose $T$ minimal in
$\{2^jT_0:j\ge1\}$ such that
\begin{equation}\label{eq:weighted-many}
\mathcal N(T)>A(1+T^d).
\end{equation}
By minimality,
\begin{equation}\label{eq:weighted-predecessor}
\mathcal N(T/2)\le A\bigl(1+(T/2)^d\bigr).
\end{equation}

Define
\[
\mathcal D=
\left\{
(\delta_1,\ldots,\delta_k)\in\cO_K^k\setminus\{0\}:
\|\delta_i\|_\infty\le\frac{2T}{\nu_i}
\text{ for }1\le i\le k
\right\}.
\]

Applying \eqref{eq:box-doubling} twice gives, for $1\le i\le k$,
\[
N_K(2T/\nu_i)
\le C_4^2N_K(T/(2\nu_i)).
\]
Multiplying these inequalities yields
\[
\mathcal N(2T)
\le
C_4^{2k}\mathcal N(T/2).
\]

Put $X=(T/2)^d$.  The lower bound in
Lemma~\ref{lem:box-count}, together with
\eqref{eq:weighted-predecessor}, gives
\[
c_4^k\frac{X^k}{V^d}
\le
c_4^k\prod_{i=1}^k\left(1+\frac{X}{\nu_i^d}\right)
\le \mathcal N(T/2)\le
A(1+X).
\]
If $X<1$, then $1+X<2$.  If $X\ge1$, then $1+X\le2X$, and hence
\[
X^{k-1}
\le
\frac{2A}{c_4^k}V^d,
\qquad
X\le
\left(\frac{2A}{c_4^k}\right)^{1/(k-1)}
V^{d/(k-1)}.
\]
In either case,
\[
1+X
\le
\left[
2+
\left(
\frac{2A}{c_4^k}
\right)^{1/(k-1)}
\right]
\left[1+V^{d/(k-1)}\right].
\]
Therefore
\[
\mathcal N(2T)
\le
A\,C_4^{2k}
\left[
2+
\left(
\frac{2A}{c_4^k}
\right)^{1/(k-1)}
\right]
\left[1+V^{d/(k-1)}\right].
\]
Since $\mathcal N(2T)$ counts all tuples satisfying the coordinate bounds,
including the zero tuple,
\[
|\mathcal D|=\mathcal N(2T)-1.
\]
Thus part~(i) holds with
\[
C_5
=
A\,C_4^{2k}
\left[
2+
\left(
\frac{2A}{c_4^k}
\right)^{1/(k-1)}
\right].
\]

Now let $y_1,\ldots,y_k\in\cO_K$ satisfy the bounds in the statement, and define
\[
\mathcal A_T=
\left\{
(a_1,\ldots,a_k)\in\cO_K^k:
\|a_i\|_\infty\le\frac{T}{\nu_i}
\text{ for }1\le i\le k
\right\}.
\]
Then $|\mathcal A_T|=\mathcal N(T)$.  Define
\[
\Theta_y:\mathcal A_T\longrightarrow\cO_K,
\qquad
\Theta_y(a_1,\ldots,a_k)=\sum_{i=1}^k a_i y_i.
\]
For every archimedean place $v$,
\[
|\Theta_y(a_1,\ldots,a_k)|_v
\le
\sum_{i=1}^k|a_i|_v|y_i|_v
\le
kMT.
\]
Thus $\|\Theta_y(a)\|_\infty\le kMT$ for $a\in\mathcal A_T$.  By
Lemma~\ref{lem:box-count} and \eqref{eq:weighted-many},
\[
\begin{aligned}
|\Theta_y(\mathcal A_T)|
&\le N_K(kMT)\\
&\le
C_4
\max\{1,(kM)^d\}(1+T^d)\\
&=
\frac{A}{2}(1+T^d)
<
\mathcal N(T)
=
|\mathcal A_T|.
\end{aligned}
\]
Hence $\Theta_y$ is not injective.  Take distinct
$a,b\in\mathcal A_T$ with $\Theta_y(a)=\Theta_y(b)$ and put $\delta_i=a_i-b_i$ for $1\le i\le k$.
Then $(\delta_1,\ldots,\delta_k)\ne0$,
$\|\delta_i\|_\infty\le2T/\nu_i$, and
$\delta_1y_1+\cdots+\delta_ky_k=0$, so
$(\delta_1,\ldots,\delta_k)\in\mathcal D$.
\end{proof}

\subsection{Proofs of the residue and product lemmas}\label{subsec:residue-product-proofs}
\begin{proof}[Proof of Lemma~\ref{lem:residue}]

Let $c_4$ be the lower-bound constant from Lemma~\ref{lem:box-count}, and put
$c_7=4^{-d}c_4$.  Suppose first that
$R\le\frac12q^{1/d}$, and let $S(R,\mathfrak p)$ be the image
modulo $\mathfrak p$ of
\[
\{a\in\cO_K:\|a\|_\infty\le R/2\}.
\]
If two distinct elements of this box have the same image modulo
$\mathfrak p$, their nonzero difference $\alpha$ lies in $\mathfrak p$.
Thus $(\alpha)\subseteq\mathfrak p$, and
\[
q\le |\Nm_{K/\mathbb Q}(\alpha)|\le\|\alpha\|_\infty^d
\le R^d\le2^{-d}q<q,
\]
a contradiction.  Thus reduction is injective, and Lemma~\ref{lem:box-count}
gives
\[
|S(R,\mathfrak p)|
=N_K(R/2)
\ge c_4\,2^{-d}R^d
\ge c_7\min\{q,R^d\}.
\]

Suppose now that $R>\frac12q^{1/d}$.  Put
$R_*=\frac12q^{1/d}$ and let $S(R,\mathfrak p)$ be the image
modulo $\mathfrak p$ of
\[
\{a\in\cO_K:\|a\|_\infty\le R_*/2\}.
\]
Reduction is injective on this smaller box because any nonzero difference $\alpha\in\mathfrak p$ would satisfy
\[
q\le|\Nm_{K/\mathbb Q}(\alpha)|\le R_*^d=2^{-d}q<q.
\]
Since Lemma~\ref{lem:box-count} holds for every nonnegative radius,
\[
|S(R,\mathfrak p)|
=N_K(R_*/2)
\ge c_4\,2^{-d}R_*^d
=c_7q
\ge c_7\min\{q,R^d\}.
\]
Every class in $S(R,\mathfrak p)$ has a representative of norm at most
$R_*/2<R/2$.
\end{proof}

\begin{proof}[Proof of Lemma~\ref{lem:product}]

Let $c_7$ be the constant from Lemma~\ref{lem:residue}.  Choose $0<c_8<1$ satisfying
\[
c_8\le\frac{1}{8nC_2},
\qquad
c_8<\frac{c_7^{\,n/d}}{8nC_2},
\]
and, for every $1\le m<n$,
\[
c_8^{n-m}
<
c_7^n(8nC_2)^{-d(n-m+1)}.
\]
Put
\[
C_8=\max\{2,2c_7^{-n}\}.
\]
Since $\lambda_i\le\lambda_n\le c_8\le(8nC_2)^{-1}$,
\[
R_i\ge R_n\ge1
\qquad(1\le i\le n).
\]
By Lemma~\ref{lem:residue},
\[
\prod_{i=1}^n|S_i|
\ge
c_7^n\prod_{i=1}^n\min\{q,R_i^d\}.
\]
If $R_n^d\ge q$, then
\[
\prod_{i=1}^n|S_i|
\ge c_7^nq^n
\ge2q^{n-1}
>q^{n-1}.
\]

Assume now that $R_n^d<q$, and put
\[
m=\min\{i:R_i^d<q\}.
\]
Since $R_1\ge\cdots\ge R_n$,
\[
\prod_{i=1}^n|S_i|
\ge
c_7^nq^{m-1}\prod_{i=m}^nR_i^d.
\]
If $m=n$, then
\[
c_7^nR_n^d
\ge
c_7^n(8nC_2c_8)^{-d}>1,
\]
so $\prod_i|S_i|>q^{n-1}$.

Suppose $m<n$.  Then
\[
\begin{aligned}
c_7^n\prod_{i=m}^nR_i^d
&\ge
c_7^n(8nC_2)^{-d(n-m+1)}
(\lambda_m\cdots\lambda_n)^{-d}\\
&>
c_8^{n-m}(\lambda_m\cdots\lambda_n)^{-d}.
\end{aligned}
\]
Since
\[
\Psi_F(1)
\le
\Psi_{F,m}(1)
=
(\lambda_m\cdots\lambda_n)^{-1/(n-m)}
\]
and $q\le c_8\Psi_F(1)^d$, we have
\[
q^{n-m}
\le
c_8^{n-m}\Psi_F(1)^{d(n-m)}
\le
c_8^{n-m}(\lambda_m\cdots\lambda_n)^{-d}.
\]
Hence
\[
\prod_{i=1}^n|S_i|
>q^{m-1}q^{n-m}=q^{n-1}.
\]
\end{proof}

\subsection{Proof of the prime-ideal interval lemma}\label{subsec:prime-interval-proof}
\begin{proof}
If $K=\mathbb Q$, Bertrand's postulate applied to $\lfloor X\rfloor$ gives
a rational prime $p$ with $X<p<2X$; thus one may take $C_9=2$.
Assume now that $d\ge2$.  Garcia and Lee
\cite{GarciaLee}*{Theorem~1, (A1)--(A2)}
give an effectively computable constant $U>0$, depending only on
$d$ and $|D_K|$, such that, for every $x\ge2$,
\[
\left|
\sum_{\substack{0\ne\mathfrak q\subset\cO_K\ \mathrm{prime}\\
\Nm\mathfrak q\le x}}
\frac{\log\Nm\mathfrak q}{\Nm\mathfrak q}
-\log x
\right|\le U.
\]
Put $C_9=\exp(2U+1)$.  Subtracting the estimates at $C_9X$ and $X$ gives
\[
\sum_{\substack{\mathfrak q\ \mathrm{prime}\\
X<\Nm\mathfrak q\le C_9X}}
\frac{\log\Nm\mathfrak q}{\Nm\mathfrak q}
\ge \log C_9-2U
=1>0.
\]
Hence there is a nonzero prime ideal $\mathfrak p$ with
$X<\Nm\mathfrak p\le C_9X$.
\end{proof}

\Addresses

\begin{bibdiv}
\begin{biblist}
\bib{Ange}{article}{author={Ange, Thomas},title={Le th\'eor\`eme de Schanuel dans les fibr\'es ad\'eliques hermitiens},date={2014},journal={Manuscripta Math.},volume={144},number={3-4},pages={565--608}}
\bib{BalkoCibulkaValtr}{article}{author={Balko, Martin},author={Cibulka, Josef},author={Valtr, Pavel},title={Covering lattice points by subspaces and counting point--hyperplane incidences},date={2019},journal={Discrete Comput. Geom.},volume={61},number={2},pages={325--354}}
\bib{BHPT}{article}{author={B{\'a}r{\'a}ny, Imre},author={Harcos, Gergely},author={Pach, J{\'a}nos},author={Tardos, G{\'a}bor},title={Covering lattice points by subspaces},date={2001},journal={Period. Math. Hungar.},volume={43},number={1-2},pages={93--103}}
\bib{BezdekHausel}{incollection}{author={Bezdek, K\'aroly},author={Hausel, Tam\'as},title={On the number of lattice hyperplanes which are needed to cover the lattice points of a convex body},book={title={Intuitive Geometry (Szeged, 1991)},series={Colloq. Math. Soc. J\'anos Bolyai},volume={63},publisher={North-Holland},address={Amsterdam},date={1994}},pages={27--31}}
\bib{BezdekLitvak}{article}{author={Bezdek, K\'aroly},author={Litvak, Alexander E.},title={Covering convex bodies by cylinders and lattice points by flats},date={2009},journal={J. Geom. Anal.},volume={19},number={2},pages={233--243}}
\bib{BombieriGubler}{book}{author={Bombieri, Enrico},author={Gubler, Walter},title={Heights in {D}iophantine geometry},series={New Mathematical Monographs},volume={4},publisher={Cambridge University Press},address={Cambridge},date={2006}}
\bib{BugeaudGyory}{article}{author={Bugeaud, Yann},author={Gy{\H{o}}ry, K{\'a}lm{\'a}n},title={Bounds for the solutions of unit equations},date={1996},journal={Acta Arith.},volume={74},number={1},pages={67--80}}
\bib{Cassels}{book}{author={Cassels, J.~W.~S.},title={An introduction to the geometry of numbers},series={Grundlehren der mathematischen Wissenschaften},volume={99},publisher={Springer-Verlag},address={Berlin--G{\"o}ttingen--Heidelberg},date={1959},note={Reprinted in Classics in Mathematics, Springer, Berlin, 1997}}
\bib{ChristensenGubler}{article}{author={Christensen, Christian},author={Gubler, Walter},title={Der relative Satz von Schanuel},date={2008},journal={Manuscripta Math.},volume={126},number={4},pages={505--525}}
\bib{FukshanskyAdditional}{article}{author={Fukshansky, Lenny},title={Siegel's lemma with additional conditions},date={2006},journal={J. Number Theory},volume={120},number={1},pages={13--25},doi={10.1016/j.jnt.2005.11.009}}
\bib{GarciaLee}{article}{author={Garcia, Stephan~Ramon},author={Lee, Ethan~Simpson},title={Unconditional explicit Mertens' theorems for number fields and Dedekind zeta residue bounds},date={2022},journal={Ramanujan J.},volume={57},number={3},pages={1169--1191}}
\bib{Gaudron2009}{article}{author={Gaudron, {\'E}ric},title={G\'eom\'etrie des nombres ad\'elique et lemmes de Siegel g\'en\'eralis\'es},date={2009},journal={Manuscripta Math.},volume={130},number={2},pages={159--182}}
\bib{Gaudron}{incollection}{author={Gaudron, {\'E}ric},title={Chapter II: Minima and slopes of rigid adelic spaces},book={title={Arakelov geometry and {D}iophantine applications},editor={Peyre, Emmanuel},editor={R{\'e}mond, Ga{\"e}l},series={Lecture Notes in Mathematics},volume={2276},publisher={Springer},address={Cham},date={2021}},pages={37--76}}
\bib{Neukirch}{book}{author={Neukirch, J{\"u}rgen},title={Algebraic number theory},series={Grundlehren der mathematischen Wissenschaften},publisher={Springer-Verlag},address={Berlin},date={1999},volume={322},translator={Schappacher, Norbert}}
\bib{RoyThunder}{article}{author={Roy, Damien},author={Thunder, Jeffrey~L.},title={An absolute Siegel's lemma},date={1996},journal={J. Reine Angew. Math.},volume={476},pages={1--26},note={Addendum and erratum, J. Reine Angew. Math. 508 (1999), 47--51}}
\bib{Schanuel}{article}{author={Schanuel, Stephen Hoel},title={Heights in number fields},date={1979},journal={Bull. Soc. Math. France},volume={107},pages={433--449}}
\bib{SudakovTomon}{article}{author={Sudakov, Benny},author={Tomon, Istv\'an},title={Evasive sets, covering by subspaces, and point--hyperplane incidences},date={2024},journal={Discrete Comput. Geom.},volume={72},number={3},pages={1333--1347}}
\end{biblist}
\end{bibdiv}
\end{document}